\documentclass[a4paper,12pt]{amsart}
\usepackage[margin=2.8cm]{geometry}
 \usepackage{amsmath,amscd}
\usepackage{fullpage}
\usepackage[applemac]{inputenc}
\usepackage{eurosym}
\newtheorem{theorem}{Theorem}[section]
\newtheorem{rem} [theorem] {Remark}

\newtheorem{definition}[theorem]{Definition}
\newcommand{\ovprt}{\overline{\partial}}
\newcommand{\ovli}{\overline}
\newcommand{\dquer}{\overline\partial}
\newcommand{\dbar}{\overline\partial}

\numberwithin{equation}{section}

\title{ Sobolev spaces for the weighted  $\ovprt$-Neumann operator}

\author{ Friedrich Haslinger}

\thanks{Partially supported by the FWF-grant  P28154.}

 \address{ F. Haslinger: Fakult\"at  f\"ur Mathematik, Universit\"at Wien,
Oskar-Morgenstern-Platz 1, A-1090 Wien, Austria}
\email{ friedrich.haslinger@univie.ac.at}


\begin{document}

\maketitle

\begin{abstract} ~\ 
We discuss compactness of the $\dbar$-Neumann operator in the setting of weighted $L^2$-spaces on $\mathbb C^n.$
In addition we describe an approach to obtain the  compactness estimates for the $\dbar$-Neumann operator. For this purpose we have to define appropriate weighted Sobolev spaces and prove an appropriate Rellich - Kondrachov lemma.
  
\end{abstract}

\section{Introduction}

Let  $\Omega $ be a bounded open set in $\mathbb{R}^n,$ and $k$  a nonnegative integer. We denote by
$W^k(\Omega)$
 the Sobolev space
$$W^k(\Omega) = \{ f\in L^2(\Omega ) \, : \, \partial^\alpha f \in L^2(\Omega ) ,  \, |\alpha |\le k \},$$
where the derivatives are taken in the sense of distributions and endow the space with the norm
\begin{equation*}
\|f\|_{k,\Omega} = \left ( \sum_{|\alpha |\le k} \int_\Omega |\partial^\alpha f |^2 \,d\lambda \right )^{1/2},
\end{equation*}
where $\alpha =(\alpha_1, \dots ,\alpha_n)$ is a multiindex , $|\alpha |=\sum_{j=1}^n \alpha_j$ and 
$$\partial^\alpha f =\frac{\partial^{|\alpha |}f}{\partial x_1^{\alpha_1}\dots \partial x_n^{\alpha_n}}.$$
$W^k(\Omega)$ is a Hilbert space.
 If $\Omega \subset \mathbb{R}^n \, , \, n\ge 2,$  is a bounded domain with a $\mathcal{C}^1$ boundary, the Rellich-Kondrachov lemma says that for $n>2$ one has
$$W^1(\Omega ) \subset L^r(\Omega) \ , \ r\in [1, 2n/(n-2))$$
and that the imbedding is also compact; for $n=2$ one can take $r\in [1,\infty)$ (see for instance \cite{AF}), in particular, there exists a constant $C_r$ such that 
\begin{equation}\label{eq: sob1}
\|f\|_r \le C_r \|f\|_{1,\Omega},
\end{equation}
for each $f\in W^1(\Omega),$ where
$$\|f\|_r =  \left ( \int_\Omega |f|^r \, d\lambda \right )^{1/r}.$$
Now let $\Omega \subseteq \mathbb{C}^n ( \cong \mathbb{R}^{2n} )$ be a smoothly bounded pseudoconvex domain.
 We consider the 
$\ovprt $-complex 
\begin{equation}\label{eq: dbarcomplex1}
L^2(\Omega )\overset{\ovprt }\longrightarrow L^2_{(0,1)}(\Omega)
\overset{\ovprt }\longrightarrow \dots \overset{\ovprt }\longrightarrow
L^2_{(0,n)}(\Omega)\overset{\ovprt }\longrightarrow 0\, ,  
\end{equation}
where $L^2_{(0,q)}(\Omega)$ denotes the space of $(0,q)$-forms on $\Omega$ with
coefficients in $L^2(\Omega).$ The $\ovprt $-operator on $(0,q)$-forms is given by
\begin{equation}\label{eq: deriv1}
\ovprt \left ( \sum_J\,^{'}  a_J \, d\ovli z_J \right )=
\sum_{j=1}^n \sum_J\,^{'}\  \frac{\partial a_J}{\partial \ovli z_j}d\ovli z_j\wedge
d\ovli z_J,
\end{equation}
where $\sum  ^{'} $ means that the sum is only taken over strictly increasing multi-indices $J.$

The derivatives are taken in the sense of distributions, and the domain of $\ovprt $
consists of those $(0,q)$-forms for which the right hand side belongs to
$L^2_{(0,q+1)}(\Omega).$ So $\ovprt $ is a densely defined closed operator\index{closed operator}, and
therefore has an adjoint operator from $L^2_{(0,q+1)}(\Omega)$ into
$L^2_{(0,q)}(\Omega)$ denoted by $\ovprt ^* .$ 

We consider the 
$\ovprt $-complex \index{$\ovprt $-complex}
\begin{equation}\label{eq: complex0}
L^2_{(0,q-1)}(\Omega )\underset{\underset{\ovprt^* }
\longleftarrow}{\overset{\ovprt }
{\longrightarrow}} L^2_{(0,q)}(\Omega )
\underset{\underset{\ovprt^* }
\longleftarrow}{\overset{\ovprt }
{\longrightarrow}} L^2_{(0,q+1)}(\Omega ),
\end{equation}
for $1\le q \le n-1.$

We remark that a $(0,q+1)$-form $u=\sum_{J}^{'} u_J\,d\ovli z_J$ belongs to $\mathcal{C}^\infty_{(0,q+1)}(\ovli \Omega ) \cap {\text{dom}}(\ovprt ^*)$ if and only if
\begin{equation}\label{eq: dom2}
\sum_{k=1}^n u_{kK} \, \frac{\partial r}{\partial z_k} =0
\end{equation}
on $b\Omega$ for all $K$ with $|K|=q,$ where $r$ is a defining function of  $\Omega$ with $|\nabla r(z)|=1$ on the boundary $b\Omega.$ (see for instance \cite{Str})

The complex Laplacian 
$\Box = \ovprt \, \ovprt ^* + \ovprt ^*\,  \ovprt $\index{$\Box$},  defined on the domain
$${\text{dom}}(\Box) = \{ u\in L^2_{(0,q)}(\Omega ) : u\in {\text{dom}}(\ovprt ) \cap  {\text{dom}}(\ovprt^*) ,  \ovprt u\in  {\text{dom}}(\ovprt^*)  ,  \ovprt^* u \in  {\text{dom}}(\ovprt ) \}$$
acts as
an unbounded, densely defined, closed and self-adjoint  operator\index{self-adjoint operator} on 
$L^2_{(0,q)}(\Omega ),$ for $ 1\le q \le n,$ which means that  $\Box = \Box^*$ and ${\text{dom}}(\Box ) =  {\text{dom}}(\Box^*).$  

Note that
\begin{equation}\label{eq: diri1}
(\Box u,u)=( \ovprt \, \ovprt ^* u+ \ovprt ^* \, \ovprt u,u)=\| \ovprt u \|^2 + \| \ovprt ^* u \|^2,
\end{equation}
for $u\in {\text{dom}}(\Box ).$

If $\Omega $ is
  a smoothly bounded pseudoconvex domain in $\mathbb{C}^n,$
  the so-called basic estimate says that
  
 \begin{equation}\label{eq: diri2}
 \| \ovprt u \|^2 + \| \ovprt ^* u \|^2 \ge c \, \|u\|^2,  
 \end{equation}
  for each $u\in {\text{dom}}(\ovprt) \cap {\text{dom}}(\ovprt^* ) , \ c>0.$

 This estimate implies that  $ \Box :  {\text{dom}}(\Box) \longrightarrow L^2_{(0,q)}(\Omega )$\index{$\Box$}
 is bijective and has a bounded inverse 
 $$N_{(0,q)}:  L^2_{(0,q)}(\Omega ) \longrightarrow {\text{dom}}(\Box). $$
 $N_{(0,q)}$ is called $\ovprt$-Neumann operator\index{$\ovprt$-Neumann operator}.
In addition 
  \begin{equation}\label{eq: cont5}
 \|N_{(0,q)} u\| \le \frac{1}{c} \, \|u\|.
 \end{equation}
Hence the $\ovprt$-Neumann operator $N_{(0,q)}$ is continuous from $L^2_{(0,q)}(\Omega )$ into itself. Compactness of the $\dbar$-Neumann operator is relevant for a number of circumstances (\cite{Str}). From the point of view of the $L^2$-Sobolev theory of the $\dbar$-Neumann operator, an important application of compactness is that it implies global regularity. Kohn and Nirenberg (\cite{KN}) proved that compactness of $N_{(0,q)}$ on $L^2_{(0,q)}(\Omega )$ implies compactness (in particular, continuity) of $N_{(0,q)}$ from the Sobolev spaces $W^s_{(0,q)}(\Omega )$ into itself for all $s\ge 0,$ see also \cite{Str}.  For this result the Rellich - Kondrachov lemma is important, it holds as $\Omega$ is a bounded domain.

\vskip 0.3 cm

The aim of this paper is to study similar properties for the weighted $\dbar$-Neumann operator on $\mathbb C^n.$

Let $\varphi : \mathbb{C}^n \longrightarrow \mathbb{R}$ be a plurisubharmonic $\mathcal{C}^2$-function and
let 
$$L^2( \mathbb{C}^n, e^{-\varphi}) = \{ g: \mathbb{C}^n \longrightarrow \mathbb{C} \ {\text{measurable}} \,:
\|g\|^2_\varphi =(g,g)_\varphi = \int_{\mathbb{C}^n} |g|^2 e^{-\varphi}\, d\lambda < \infty \}.$$
Let $1\le q \le n$ and
$$f= \sum_{|J|=q}\, ' \, f_J\,d\overline z_J ,$$
where the sum is taken only over increasing multiindices $J=(j_1, \dots , j_q)$ and $d\overline z_J = d\overline z_{j_1} \wedge \dots \wedge d\overline z_{j_q}$ and $f_J\in L^2(\mathbb C^n, e^{-\varphi}).$

We write $f\in L^2_{(0,q)}(\mathbb C^n, e^{-\varphi})$ and define 
$$\ovprt f = \sum_{|J|=q}\, ' \, \sum_{j=1}^n \frac{\partial f_J}{\partial \overline z_j}\, d\overline z_j \wedge d\overline z_J$$
for $1\le q \le n-1$ and 
$${\text{dom}}(\ovprt) = \{ f \in L^2_{(0,q)}(\mathbb C^n, e^{-\varphi}) \, : \, \ovprt f \in L^2_{(0,q+1)}(\mathbb C^n, e^{-\varphi}) \} .$$
In this way $\ovprt$ becomes a densely defined closed operator and its adjoint $\ovprt^*_\varphi$ depends on the weight $\varphi.$

We consider the weighted $\ovprt$-complex
\begin{equation*}
L^2_{(0,q-1)}(\mathbb{C}^n , e^{-\varphi} )\underset{\underset{\ovprt_\varphi^* }
\longleftarrow}{\overset{\ovprt }
{\longrightarrow}} L^2_{(0,q)}(\mathbb{C}^n , e^{-\varphi} )
\underset{\underset{\ovprt_\varphi^* }
\longleftarrow}{\overset{\ovprt }
{\longrightarrow}} L^2_{(0,q+1)}(\mathbb{C}^n , e^{-\varphi} )
\end{equation*}
and we set 
$$\Box_{\varphi}^{(0,q)}= \ovprt\, \ovprt_\varphi^* +	 \ovprt_\varphi^* \ovprt,$$
where
$${\text{dom}}(\Box_\varphi^{(0,q)}) = \{ u \in {\text{dom}}(\ovprt) \cap {\text{dom}}(\ovprt^*_\varphi): \ovprt u \in  {\text{dom}}(\ovprt^*_\varphi), \ovprt^*_\varphi u \in  {\text{dom}}(\ovprt)\}.$$

It turns out that $\Box_{\varphi}^{(0,q)}$ is a densely defined, non-negative self-adjoint operator, which has a uniquely determined self-adjoint square root $(\Box_{\varphi}^{(0,q)})^{1/2}.$ The domain of
$(\Box_{\varphi}^{(0,q)})^{1/2})$ coincides with ${\text{dom}}(\ovprt) \cap {\text{dom}}(\ovprt^*_\varphi),$ which is also the domain of the corresponding quadratic form
$$Q_\varphi (u,v):=(\ovprt u, \ovprt v)_\varphi + (\ovprt^*_\varphi u, \ovprt^*_\varphi v)_\varphi,$$
see for instance \cite{Dav}. 

Next we consider the Levi matrix 
$$M_\varphi = \left ( \frac{\partial^2 \varphi}{\partial z_j \partial \overline z_k} \right )_{j,k=1}^n$$

and suppose that the lowest eigenvalue $\mu_\varphi$ of $M_\varphi$ satisfies
\begin{equation} \label{perss}
\liminf_{|z| \to \infty} \mu_\varphi (z)>0.
\end{equation}

\eqref{perss} implies that $\Box_\varphi ^{(0,1)}$ is injective and that the bottom of the essential spectrum
$\sigma_e(\Box_\varphi ^{(0,1)})$ is positive (Persson's Theorem), see \cite{HaHe}.
Now it follows that $\Box_\varphi ^{(0,1)}$ has a bounded inverse,  which we denote by 
$$N_\varphi ^{(0,1)}: L^2_{(0,1)}(\mathbb{C}^n , e^{-\varphi} ) \longrightarrow L^2_{(0,1)}(\mathbb{C}^n , e^{-\varphi} ).$$
Using the square root of  $N_\varphi ^{(0,1)}$ we get the basic estimates
\begin{equation}\label{coerc}
\|u\|^2_\varphi \le C ( \|\ovprt u \|^2_\varphi + \|\ovprt^*_\varphi  u\|^2_\varphi ),
\end{equation}
for all $u\in {\text{dom}}(\ovprt) \cap {\text{dom}}(\ovprt^*_\varphi).$

Now we will study compactness of the weighted  $\dbar$-Neumann operator $N_\varphi ^{(0,1)}.$ For this purpose we will use  the description of compact subsets in $L^2$-spaces, as it is done in \cite{Has10} Chapter 11, to derive  a sufficient condition for compactness in terms of the weight function. It turns out that compactness of the $\dbar$-Neumann operator $N_\varphi ^{(0,1)}$ is equivalent to compactness of the embedding of a certain complex Sobolev space into $L^2_{(0,1)}(\mathbb{C}^n , e^{-\varphi}).$

\begin{definition} Let

$$\mathcal{W}^{Q_\varphi}= \{ u\in L^2_{(0,1)}(\mathbb{C}^n , e^{-\varphi}) \ : \  u\in {\text{dom}}(\ovprt) \cap {\text{dom}}(\ovprt^*_\varphi) \}$$ with norm
\begin{equation}\label{com9}
\| u\|_{Q_\varphi}  =  (\| \ovprt u \|^2_\varphi + \| \ovprt_\varphi ^* u\|^2_\varphi )^{1/2}.
\end{equation}
\end{definition}

So $\mathcal{W}^{Q_\varphi}$ is the form domain  of $Q_\varphi .$

\begin{theorem}\label{sec: embed}Suppose that the weight function  $\varphi$ is  plurisubharmonic  and that the lowest eigenvalue $\mu_{\varphi}$ of the Levi - matrix $M_{\varphi }$ satisfies
\begin{equation}\label{eq: com10}
\lim_{|z|\rightarrow \infty}\mu_\varphi(z)  = +\infty\, . 
\end{equation}
Then the embedding 
\begin{equation}\label{eq: com11}
j_\varphi : \mathcal{W}^{Q_\varphi } \hookrightarrow  L^2_{(0,1)}(\mathbb{C}^n, e^{-\varphi} )
\end{equation}
 is compact. Consequently, the $\ovprt$-Neumann operator $N_\varphi^{(0,1)}$ is compact.
\end{theorem}
This result can be seen as a Rellich Kondrachov lemma for Sobolev spaces defined by complex derivatives.
Notice that 
$$N_\varphi^{(0,1)} : L^2_{(0,1)}(\mathbb{C}^n , e^{-\varphi}) \longrightarrow L^2_{(0,1)}(\mathbb{C}^n , e^{-\varphi}) $$
can be written in the form 
$$N_\varphi^{(0,1)} = j_\varphi \circ j_\varphi ^* \ , $$
where 
$$j_\varphi ^* : L^2_{(0,1)}(\mathbb{C}^n , e^{-\varphi}) \longrightarrow \mathcal{W}^{Q_\varphi}$$
is the adjoint operator to $j_\varphi,$ see \cite{Has10} Section 6.2, or \cite{Str} Section 2.8. 

\vskip 0.3 cm

It is now clear that $N_\varphi^{(0,1)} $ is compact if and only if $j_\varphi$ is compact. 

We have to show that the unit ball in $\mathcal{W}^{Q_\varphi }$ is relatively compact in $ L^2_{(0,1)}(\mathbb{C}^n, e^{-\varphi} )$. For this purpose we use the characterization of compact subsets in $L^2$-spaces (see \cite{Has10} Chapter 11).

For $u\in \mathcal{W}^{Q_\varphi}$ we have 
$$ \| \ovprt u \|^2_\varphi + \| \ovprt_\varphi ^* u\|^2_\varphi \ge  ( M_{\varphi } u, u)_\varphi .$$

This implies
\begin{equation}\label{eq: compac}
 \| \ovprt u \|^2_\varphi + \| \ovprt_\varphi ^* u\|^2_\varphi \ge  \int_{\mathbb{C}^n}\mu_\varphi(z) \, |u(z)|^2\,  e^{-\varphi(z)}\,d\lambda (z) \ge \int_{\mathbb{C}^n \setminus \mathbb{B}_R} \mu_\varphi(z) |u(z)|^2   e^{-\varphi(z)}  d\lambda (z),
\end{equation}
 where $\mathbb{B}_R$ is the ball with center $0$ and radius $R>0.$

Consequently, assumption \eqref{eq: com10} implies that for each $\epsilon >0$ there is $R>0$ such that 

\begin{equation}\label{eq: com12}
\int_{\mathbb{C}^n \setminus \mathbb{B}_R} |u(z)|^2 e^{-\varphi (z)}\,d\lambda (z) < \epsilon,
 \end{equation}
for all $u$ in the unit ball of $\mathcal{W}^{Q_\varphi}.$ 
Also, the map $u \mapsto u|_{\mathbb B_R}$ is compact from $\mathcal{W}^{Q_\varphi }$ to $L^2_{(0,1)}(\mathbb B_R, e^{-\varphi} ),$ in view of the ellipticity of $\ovprt \oplus \ovprt^*_\varphi.$ Together with \eqref{eq: com12}, this latter fact shows that the image of a bounded set in $\mathcal{W}^{Q_\varphi }$ is pre-compact in $L^2_{(0,1)} (\mathbb C^n, e^{-\varphi}).$
\vskip 0.5 cm
In the following we will describe an approach to obtain the so-called compactness estimates for the $\dbar$-Neumann operator $N_\varphi^{(0,1)} ,$ where we follow  
\cite{Str}, Propostion 4.2.
For this purpose we have to define appropriate weighted Sobolev spaces and we need an appropriate Rellich - Kondrachov lemma.
\vskip 2 cm

\section{Weighted $L^2$-Sobolev spaces}
\vskip 2 cm 

Let $z=(z_1,\dots, z_n)=(x_1+iy_1, \dots , x_n+iy_n)\in \mathbb C^n$ and write for a multiindex 
$$\gamma = (\gamma_1, \gamma_2, \dots , \gamma_{2n-1}, \gamma_{2n})$$ and an appropriate function 
$$\partial^\gamma f= \frac{\partial^{|\gamma|}f}{\partial x_1^{\gamma_1} \partial y_1^{\gamma_2} \dots 
\partial x_n^{\gamma_{2n-1}} \partial y_n^{\gamma_{2n} }}.$$

\begin{definition}\label{sobo1}
We denote by
$W^k(\mathbb C^n)$
 the Sobolev space
$$W^k(\mathbb C^n) = \{ f\in L^2(\mathbb C^n ) \, : \, \partial^\gamma f \in L^2(\mathbb C^n ) ,  \, |\gamma |\le k \},$$
where the derivatives are taken in the sense of distributions and endow the space with the norm
\begin{equation*}
\|f\|_{k} = \left ( \sum_{|\gamma |\le k} \int_{\mathbb C^n} |\partial^\gamma f |^2 \,d\lambda \right )^{1/2}.
\end{equation*}
\end{definition}
$W^k(\mathbb C^n)$ is a Hilbert space. It is well-known that the embedding $\iota: W^1(\mathbb C^n) \hookrightarrow L^2(\mathbb C^n) $ fails to be compact. In sake of completeness we recall the easy proof: let $\psi \in \mathcal C^\infty_0(\mathbb C^n)$ be a smooth function with compact support such that ${\text{Tr}} \psi \subset B_{1/2}(0)$ and $\int_{\mathbb C^n}|\psi (z)|^2\,d\lambda (z)=1.$ For $k\in \mathbb{N}$ let $\psi_k(z)= \psi (z-\overrightarrow{k}),$ where $\overrightarrow{k} = (k, 0, \dots ,0) \in \mathbb C^n.$ Then ${\text{Tr}}\psi_k \subset 
B_{1}(\overrightarrow{k})$ and $(\psi_k)_k$ is a bounded sequence in $W^1(\mathbb C^n).$ Now let $k,m\in \mathbb N$ with $k\neq m.$ Due to the fact that $\psi_k$ and $\psi_m$ have non-overlapping supports we have
$$\| \psi_k - \psi_m \|^2 = \| \psi_k \|^2 + \| \psi_m \|^2 = 2,$$
and the sequence $(\psi_k)_k$ has no convergent subsequence in $L^2(\mathbb C^n).$

Let $U_\varphi : L^2(\mathbb C^n ) \longrightarrow L^2(\mathbb C^n, e^{-\varphi} )$ denote the isometry given by $U_\varphi (f)= f e^{\varphi/2},$ for $f\in L^2(\mathbb C^n ).$ The inverse is given by $U_{-\varphi} (g) = ge^{-\varphi/2},$ for $ g\in L^2(\mathbb C^n, e^{-\varphi} ).$
The appropriate weighted Sobolev spaces are determined as the images of $W^k(\mathbb C^n)$ under the isometry $U_\varphi.$ In the following we consider only Sobolev spaces of order $1.$ Let $f\in W^1(\mathbb C^n).$ Then $fe^{\varphi/2} , (\partial_jf)e^{\varphi/2} \in L^2(\mathbb C^n, e^{-\varphi}), $ where $\partial_j f$ denotes all first order derivatives of $f$ with respect to $x_j$ and $y_j$ for $j=1,\dots,n.$ Set $h=fe^{\varphi/2}.$ Then
\begin{align*}
\partial_j h& = (\partial_jf)e^{\varphi/2}+\frac{1}{2} f (\partial_j\varphi) e^{\varphi/2} \\
& = (\partial_jf)e^{\varphi/2} +\frac{1}{2} (\partial_j\varphi) h,
\end{align*}
which implies $(\partial_jf)e^{\varphi/2} = \partial_j h - \frac{1}{2} (\partial_j\varphi) h$ and
\vskip 0.3 cm
$U_\varphi (W^1(\mathbb C^n))= \{ h\in L^2(\mathbb C^n, e^{-\varphi}): \partial_j h - \frac{1}{2} (\partial_j\varphi) h \in L^2(\mathbb C^n, e^{-\varphi}), j=1, \dots ,2n \}.$
\vskip 0.3 cm
For reasons which will become clear later, we denote
$$W^1_0(\mathbb C^n, e^{-\varphi}):= U_\varphi (W^1(\mathbb C^n)),$$
and we endow the space $W^1_0(\mathbb C^n, e^{-\varphi})$ with the norm
$h \mapsto ( \|h\|^2_\varphi + \sum_j \|  \partial_j h - \frac{1}{2} (\partial_j\varphi) h \|^2_\varphi  )^{1/2}.$
\vskip 0.3 cm
in this way $U_\varphi : W^1(\mathbb C^n) \longrightarrow W^1_0(\mathbb C^n, e^{-\varphi})$ is again isometric and we have the following commutative diagram
\[ \begin{CD}
W^1(\mathbb C^n) @> \iota >> L^2(\mathbb C^n)  \\
@V{U_\varphi} VV @VV{U_\varphi}V \\
W^1_{0}(\mathbb C^n, e^{-\varphi}) @>>\iota_\varphi > L^2(\mathbb C^n, e^{-\varphi})
\end{CD} \]
where 
$\iota_\varphi : W^1_0(\mathbb C^n, e^{-\varphi}) \hookrightarrow L^2(\mathbb C^n, e^{-\varphi})$ 
is the canonical embeddings.
As $U_\varphi \, \iota = \iota_\varphi \, U_\varphi $ and $\iota $ fails to be compact, $\iota_\varphi $ is also not compact.

\begin{definition}\label{sobo2}
Let $\eta \in \mathbb R.$ We denote by
$W^1_{\eta}(\mathbb C^n, e^{-\varphi})$
 the Sobolev space
 \vskip 0.3 cm
$W^1_{\eta}(\mathbb C^n, e^{-\varphi}) = \{ h\in L^2(\mathbb C^n, e^{-\varphi} ) : \partial_j h - \frac{1+\eta}{2} (\partial_j\varphi) h \in L^2(\mathbb C^n, e^{-\varphi}), j=1, \dots ,2n \},$
\vskip 0.3 cm
endowed with the norm
$h \mapsto ( \|h\|^2_\varphi + \sum_j \|  \partial_j h - \frac{1+\eta}{2} (\partial_j\varphi) h \|^2_\varphi  )^{1/2}.$
\end{definition}
We use the notation 
$$X_j=\frac{\partial }{\partial x_j} - \frac{1+\eta}{2}\frac{\partial \varphi}{\partial x_j} \ {\text {and}} \ 
Y_j=\frac{\partial }{\partial y_j} -  \frac{1+\eta}{2}\frac{\partial \varphi}{\partial y_j},$$
for $j=1,\dots ,n.$ Then
$$W^1_\eta(\mathbb C^n, e^{-\varphi} )=\{ f\in L^2(\mathbb{C}^n, e^{-\varphi}) \ : X_jf,\ Y_jf \in  L^2(\mathbb{C}^n, e^{-\varphi}) , j=1,\dots ,n \},$$
with norm
$$\|f\|^2_{ \varphi , \eta}= \|f\|^2_{ \varphi }+\sum_{j=1}^n( \|X_jf\|^2_\varphi
+ \|Y_jf\|^2_\varphi) .$$

For suitable weight functions $\varphi,$ we can prove an analogous result to the Rellich Kondrachov lemma.

\begin{theorem}\label{rellich}
Suppose that $\varphi $ is a $\mathcal{C}^2$-function satisfying

\begin{equation}\label{eq: nowei}
\lim_{|z|\to \infty}(\eta^2 |\nabla \varphi (z)|^2+(1+\epsilon) \eta \, \triangle \varphi (z))= +\infty ,
\end{equation}
for some $\epsilon >0,$ where 
$$|\nabla \varphi (z)|^2= \sum_{k=1}^n \left ( \left | \frac{\partial \varphi}{\partial x_k}\right |^2+ \left | \frac{\partial \varphi}{\partial y_k}\right |^2 \right ).$$
Then the canonical embedding $\iota_{\varphi, \eta} :W^1_\eta (\mathbb C^n,e^{- \varphi})\hookrightarrow L^2(\mathbb{C}^n, e^{-\varphi}) $ is compact.
\end{theorem}
\begin{proof}
We adapt methods from  \cite{BDH} , \cite{Jo} and \cite{KM} and use the general result that an operator between Hilbert spaces is compact if and only if the image of a weakly convergent sequence is strongly convergent. 

In addition we remark that $\mathcal C^\infty_0(\mathbb C^n)$ is dense in all spaces which are involved.
For the vector fields $X_j$  and their adjoints $X_j^*$ in the weighted space $L^2(\mathbb C^n,e^{-\varphi})$ we have $X_j^*=-\frac{\partial}{\partial x_j} + \frac{1-\eta}{2}\frac{\partial \varphi}{\partial x_j}$ and
\begin{equation}\label{eq: expl}
(X_j+X^*_j)f=-\eta \frac{\partial \varphi}{\partial x_j}\, f \ {\text{and}} \ 
[X_j,X^*_j]f= -\eta \frac{\partial^2\varphi}{\partial x_j^2}\, f,
\end{equation}
for $f\in  \mathcal{C}^\infty_0(\mathbb{C}^n),$ and
\begin{equation}\label{eq: comsob1}
( [X_j,X^*_j]f,f )_\varphi=\|X^*_jf\|^2_\varphi - \|X_jf\|^2_\varphi ,
\end{equation}
\begin{equation}\label{eq: comsob2}
\|(X_j+X^*_j)f\|^2_\varphi \le (1+1/\epsilon)\|X_jf\|^2_\varphi + (1+\epsilon)\|X^*_jf\|^2_\varphi 
\end{equation}
for each $\epsilon>0,$ where we used the inequality
$$|a+b|^2 \le |a|^2 + |b|^2 + 1/\epsilon \, |a|^2 + \epsilon \, |b|^2.$$
Similar relations hold for the vector fields $Y_j.$
Now we set
$$\Psi (z)=\eta^2|\nabla \varphi (z)|^2+(1+\epsilon)\eta \triangle \varphi (z).$$

By \eqref{eq: expl}, \eqref{eq: comsob1} and \eqref{eq: comsob2}, it follows that
$$( \Psi f,f )_\varphi \le 
(2+\epsilon +1/\epsilon)\sum_{j=1}^n( \|X_jf\|^2_\varphi
+ \|Y_jf\|^2_\varphi) .$$
Since $\mathcal{C}^\infty_0(\mathbb{C}^n)$ is dense in $W^1_\eta(\mathbb C^n, e^{-\varphi})$ by definition, this inequality holds for all $f\in W^1_\eta(\mathbb C^n, e^{-\varphi}).$

If $(f_k)_k$ is a sequence in $W^1_\eta(\mathbb C^n, e^{-\varphi})$ converging weakly to $0,$ then $(f_k)_k$ is a bounded sequence in $W^1_\eta(\mathbb C^n, e^{-\varphi})$ and our  assumption implies that 
$$\Psi (z)=\eta^2 |\nabla \varphi (z)|^2+(1+\epsilon) \eta \triangle \varphi (z)$$
is positive in a neighborhood of $\infty $. So we obtain
\begin{eqnarray*}
 \int\limits_{\mathbb{C}^n}|f_k(z)|^2e^{-\varphi (z)}\,d\lambda (z) & \le &
 \int\limits_{|z|< R}|f_k(z)|^2e^{-\varphi (z)}\,d\lambda (z)\\
 & + &
 \int\limits_{|z|\ge R} \frac{\Psi (z) |f_k(z)|^2}{\inf \{\Psi (z) \, : \, |z|\ge R\}} \, e^{-\varphi (z)}\,d\lambda (z)\\
 &\le & C_{\varphi , R}\, \|f_k\|^2_{L^2(B(0,R))}+ \frac{C_\epsilon \, \|f_k\|^2_{\varphi, \eta}}{\inf \{\Psi (z) \, : \, |z|\ge R\}}.
\end{eqnarray*}
Notice that in the last estimate the expression $\Psi (z)$ plays a similar role as $\mu_{\varphi} (z)$ in \eqref{eq: compac}.
It is now easily seen that the sequence $(f_k)_k$ converges also weakly to zero in $W^1(B(0,R)).$
Hence the assumption and the fact that the embedding
$$W^1(B(0,R)) \hookrightarrow L^2(B(0,R))$$
 is compact (classical Rellich Kondrachov Lemma, see for instance \cite{AF}) show that $(f_k)_k$ tends to $0$ in $L^2(\mathbb{C}^n, e^{-\varphi}).$

\end{proof}
\begin{rem}
If $\eta =0,$ we get the case corresponding to $W^1(\mathbb C^n), $
whereas $\eta =-1$ corresponds to the Sobolev space of all functions $h\in L^2(\mathbb C^n, e^{-\varphi})$ such that all derivatives of order $1$ satisfy $\partial_j h \in  L^2(\mathbb C^n, e^{-\varphi});$
in this case the higher order Sobolev spaces are defined as the spaces of all functions $h\in L^2(\mathbb C^n, e^{-\varphi})$ such that all derivatives of order $k\ge 1$ belong to $L^2(\mathbb C^n, e^{-\varphi}).$
\end{rem}

From Theorem \ref{rellich} we can also derive compactness for embeddings in Sobolev spaces without weights.
For this purpose we define

\begin{definition}\label{sobo7}
Let $\eta \in \mathbb R.$ We define

$ W^1_\eta (\mathbb C^n, \nabla \varphi) := \{ f\in L^2(\mathbb C^n) : \partial_j f - \frac{\eta}{2}(\partial_j \varphi)f \in L^2(\mathbb C^n), j=1, \dots, 2n\}.$
\end{definition}
Then $U_\varphi : W^1_\eta (\mathbb C^n, \nabla \varphi) \longrightarrow W^1_\eta (\mathbb C^n,e^{- \varphi})$ is an isometry. We consider the canonical embedding $\iota_\eta : W^1_\eta (\mathbb C^n, \nabla \varphi) \hookrightarrow  L^2(\mathbb C^n)$ and we have the following commutative diagram
\[ \begin{CD}
 W^1_\eta (\mathbb C^n, \nabla \varphi) @> \iota_\eta >> L^2(\mathbb C^n)  \\
@V{U_\varphi} VV @VV{U_\varphi}V \\
W^1_{\eta}(\mathbb C^n, e^{-\varphi}) @>>\iota_{\varphi, \eta} > L^2(\mathbb C^n, e^{-\varphi})
\end{CD} \]
Hence the condition \eqref{eq: nowei} implies that the canonical embedding 
$\iota_\eta: W^1_\eta (\mathbb C^n, \nabla \varphi) \hookrightarrow L^2(\mathbb C^n)$
is compact.

\vskip 0.4 cm
Now we return to compactness of the $\dbar$-Neumann operator $N_\varphi^{(0,1)}.$ We consider the weighted Sobolev space 

$W^1_{1}(\mathbb C^n, e^{-\varphi}) = \{ h\in L^2(\mathbb C^n, e^{-\varphi} ) : \partial_j h - (\partial_j\varphi) h \in L^2(\mathbb C^n, e^{-\varphi}), j=1, \dots ,2n \},$

and use 
$$X_j=\frac{\partial }{\partial x_j} - \frac{\partial \varphi}{\partial x_j} \ {\text {and}} \ 
Y_j=\frac{\partial }{\partial y_j} - \frac{\partial \varphi}{\partial y_j},$$
for $j=1,\dots ,n.$ Then
$$W^1_1(\mathbb C^n, e^{-\varphi} )=\{ f\in L^2(\mathbb{C}^n, e^{-\varphi}) \ : X_jf,\ Y_jf \in  L^2(\mathbb{C}^n, e^{-\varphi}) , j=1,\dots ,n \},$$
with norm
$$\|f\|^2_{ \varphi , 1}= \|f\|^2_{ \varphi }+\sum_{j=1}^n( \|X_jf\|^2_\varphi
+ \|Y_jf\|^2_\varphi) .$$

We point out that  each continuous linear functional $L$ on $W^1_1(\mathbb C^n, e^{-\varphi} )$ is represented  by
$$L(f) = \int_{\mathbb{C}^n} fg_0e^{-\varphi}\,d\lambda +
\sum_{j=1}^n   \int_{\mathbb{C}^n}((X_jf)g_j+(Y_jf)h_j)e^{-\varphi }\,d\lambda,$$
for $f\in W^1_1(\mathbb C^n, e^{-\varphi} )$ and for some $g_0,g_j,h_j\in L^2(\mathbb{C}^n, e^{-\varphi}), \, j=1,\dots,n.$ In particular, each function in $L^2(\mathbb C^n, e^{-\varphi})$ can be identified with an element
of the dual space $(W^1_1(\mathbb C^n, e^{-\varphi} ))' =: W^{-1}_1(\mathbb C^n, e^{-\varphi} ).$ We denote the norm in $ W^{-1}_1(\mathbb C^n, e^{-\varphi} )$ by $\| \, . \, \|_{\varphi, -1}.$ See \cite{Has10} Chapter 11,  for more details.

If we suppose that  $\varphi $ is a $\mathcal{C}^2$-function satisfying

\begin{equation}\label{eq: nowei1}
\lim_{|z|\to \infty}( |\nabla \varphi (z)|^2+(1+\epsilon) \, \triangle \varphi (z))= +\infty ,
\end{equation}
for some $\epsilon >0,$ then the embedding 
$$L^2_{(0,1)}(\mathbb C^n, e^{-\varphi}) \hookrightarrow  W^{-1}_{1,(0,1)}(\mathbb C^n, e^{-\varphi} )$$
is compact  by Theorem \ref{rellich} and duality. So, as in \cite{Str}, Proposition 4.2 or \cite{Has10}, Proposition 11.20, we get the compactness estimates

\begin{theorem}\label{sec: compact}~\\ 
Suppose that the weight function $\varphi$  satisfies \eqref{perss} and
$$\lim_{|z|\to \infty}( |\nabla \varphi (z)|^2+(1+\epsilon) \, \triangle \varphi (z))= +\infty ,$$
for some $\epsilon >0,$ then the following statements are equivalent.
\begin{enumerate}
  \item The $\ovprt $-Neumann operator $N_{\varphi}^{(0,1)}$ is a compact operator from $L_{(0,1)}^2(\mathbb{C}^n, e^{-\varphi})$ into itself.
  \item The embedding of the space dom\,$(\ovprt )\,\cap$
dom\,$(\ovprt_\varphi^*),$ provided with the graph norm $u\mapsto (\|u\|^2_\varphi + \|\ovprt u\|^2_\varphi + 
\|\ovprt_\varphi ^*u\|^2_\varphi)^{1/2},$ into $L^2_{(0,1)}(\mathbb{C}^n, e^{-\varphi})$ is compact.
  \item For every positive $\epsilon' $ there exists a constant $C_{\epsilon'}$ such that
  $$\|u\|_\varphi ^2 \le \epsilon' (\| \ovprt u \|_\varphi ^2 + \| \ovprt_\varphi ^*u\|_\varphi ^2) + C_{\epsilon'} \|u\|_{\varphi , -1} ^2,$$
  for all $u\in$ dom\,$(\ovprt )\,\cap\,$dom\,$(\ovprt_\varphi^*).$
\item  For every positive $\epsilon'$ there exists $R>0$ such that 
\begin{equation*} 
 \int_{\mathbb{C}^n \setminus \mathbb{B}_R} |u(z)|^2\, e^{-\varphi(z)}\,d\lambda(z) \le \epsilon' ( \|\ovprt u\|_\varphi ^2 + \| \ovprt_\varphi^*  u\|_\varphi ^2)
\end{equation*}
for all $u\in dom\,(\ovprt) \,\cap dom\,(\ovprt_\varphi^*). $

\item The operators
$$\ovprt_\varphi ^* N_{\varphi}^{(0,1)} : L_{(0,1)}^2(\mathbb{C}^n, e^{-\varphi})\cap {\text {ker}}(\ovprt) \longrightarrow L^2(\mathbb{C}^n, e^{-\varphi}) \ \ {\text {and}}$$
$$\ovprt_\varphi ^* N_{\varphi}^{(0,2)} : L_{(0,2)}^2(\mathbb{C}^n, e^{-\varphi})\cap {\text {ker}}(\ovprt) \longrightarrow L_{(0,1)}^2(\mathbb{C}^n, e^{-\varphi})$$
are both compact.
\end{enumerate}
\end{theorem}

\begin{rem}
If
$$\lim_{|z|\rightarrow \infty}\mu_\varphi(z)  = +\infty,$$
then the condition of the Rellich-Kondrachov lemma \eqref{eq: nowei1} is satisfied.

This follows from the fact that we have for the trace ${\text {tr}}(M_\varphi ) $ of the Levi - matrix 
$${\text {tr}}(M_\varphi )=\frac{1}{4}\triangle \varphi, $$
and since for any invertible $(n\times n)$-matrix $T$
$$ {\text {tr}}(M_\varphi )={\text {tr}}(TM_\varphi T^{-1}),$$
it follows that ${\text {tr}}(M_\varphi )$ equals the sum of all eigenvalues of $M_\varphi .$

We mention that for the weight $\varphi (z)=|z|^2$ the $\dquer $-Neumann operator fails to be compact (see \cite{Has10} Chapter 15), but  condition \eqref{eq: nowei1} is
satisfied.

In view of Theorem \ref{rellich} it is  clear that for any weight satisfying \eqref{perss} and \eqref{eq: nowei} for $\eta \in \mathbb R, \eta \neq 0,$ and for some $\epsilon >0,$ the restriction of the $\dbar$-Neumann operator $N_{\varphi}^{(0,1)} $ to 
$W^1_{\eta, (0,1)} (\mathbb C^n,e^{- \varphi})$ is compact as an operator from $W^1_{\eta, (0,1)} (\mathbb C^n,e^{- \varphi})$ to $L^2_{(0,1)}(\mathbb C^n, e^{-\varphi}).$
\end{rem}

 \vskip 1 cm
ACKNOWLEDGMENT: The author wishes to thank the referee for the valuable comments.

\vskip 1 cm

\bibliographystyle{amsplain}
\bibliography{mybibliography}

\providecommand{\bysame}{\leavevmode\hbox to3em{\hrulefill}\thinspace}
\providecommand{\MR}{\relax\ifhmode\unskip\space\fi MR }
\providecommand{\MRhref}[2]{%
  \href{http://www.ams.org/mathscinet-getitem?mr=#1}{#2}
}
\providecommand{\href}[2]{#2}
\begin{thebibliography}{1}

\bibitem{AF}
R.A. Adams and J.J.F. Fournier, \emph{Sobolev spaces}, Pure and Applied Math.,
  vol. 140, Academic Press, 2006.

\bibitem{BDH}
P.~Bolley, M.~Dauge, and B.~Helffer, \emph{{Conditions suffisantes pour
  l'injection compacte d'espace de Sobolev \`a poids }}, S\'eminaire \'equation
  aux d\'eriv\'ees partielles (France), Universit\'e de Nantes \textbf{1}
  (1989), 1--14.

\bibitem{Dav}
E.B. Davies, \emph{Spectral theory and differential operators}, Cambridge
  studies in advanced mathematics, vol.~42, Cambridge University Press,
  Cambridge, 1995.

\bibitem{Has10}
F.~Haslinger, \emph{{The $\ovprt$-Neumann problem and Schr\"odinger
  operators}}, de Gruyter Expositions in Mathematics 59, Walter De Gruyter,
  2014.

\bibitem{HaHe}
F.~Haslinger and B.~Helffer, \emph{{ Compactness of the solution operator to
  $\ovprt $ in weighted $L^ 2$ - spaces}}, J. of Functional Analysis
  \textbf{243} (2007), 679--697.

\bibitem{Jo}
J.~Johnsen, \emph{{On the spectral properties of Witten Laplacians, their range
  projections and Brascamp-Lieb's inequality }}, Integral Equations Operator
  Theory \textbf{36} (2000), 288--324.

\bibitem{KM}
J.-M. Kneib and F.~Mignot, \emph{{Equation de Schmoluchowski g\'en\'eralis\'ee
  }}, Ann. Math. Pura Appl. (IV) \textbf{167} (1994), 257--298.

\bibitem{KN}
J.~Kohn and L.~Nirenberg, \emph{{ Non-coercive boundary value problems}}, Comm.
  Pure and Appl. Math. \textbf{18} (1965), 443--492.

\bibitem{Str}
E.~Straube, \emph{{The $L^2$-Sobolev theory of the $\ovprt $-Neumann problem
  }}, ESI Lectures in Mathematics and Physics, EMS, 2010.

\end{thebibliography}

\end{document}